\documentclass[12pt,reqno]{amsart}

\usepackage[a4paper,bindingoffset=0.1in,left=1.2in,right=1.3in,top=1.4in,bottom=1.5in,footskip=.25in]{geometry}

\usepackage{indentfirst,amssymb,amsfonts,amsmath,amsthm}    
\usepackage{newtxtext,newtxmath}
\usepackage{setspace}
\usepackage{times}
\usepackage[utf8]{inputenc}
\usepackage[T1]{fontenc}
\usepackage{verbatim}

\usepackage{hyperref}
\hypersetup{colorlinks=true, linkcolor=blue,citecolor=blue, urlcolor=blue}
\DeclareMathOperator{\ord}{ord}
\DeclareMathOperator{\dime}{dim}
\DeclareMathOperator{\pideg}{PI-deg}

\DeclareMathOperator{\cf}{Fract}
\DeclareMathOperator{\gcdi}{gcd}

\DeclareMathOperator{\diagonal}{diag}

\DeclareMathOperator{\colu}{col}
\DeclareMathOperator{\rowa}{row}

\DeclareMathOperator{\kere}{ker}
\DeclareMathOperator{\ran}{rank}
\DeclareMathOperator{\spa}{span}
\DeclareMathOperator{\ke}{ker}
\DeclareMathOperator{\mo}{mod}
\DeclareMathOperator{\mi}{min}
\usepackage{mathtools}

\numberwithin{equation}{section}
 
\newtheorem{theom}{Theorem}[section]
\newtheorem{defin}{Definition}[section]
\newtheorem{propn}{Proposition}[section]

\newtheorem{remak}{Remark}[section] 
\newtheorem{lemma}{Lemma}[section]
\newtheorem{theo}{Theorem}[subsection]

\newtheorem{lemm}{Lemma}[subsection]
\newtheorem{rema}{Remark}[subsection]

\newtheorem{prop}{Proposition}[subsection]

\begin{document}

\setcounter{page}{1} 
\baselineskip .65cm 
\pagenumbering{arabic}

\title[Quantum Spatial Ageing Algebra]{Representations Of Quantum Spatial Ageing Algebra\\ at root of unity}
\author [Snehashis Mukherjee~And~Sanu Bera]{Snehashis Mukherjee$^1$ \and Sanu Bera$^2$}

\address {\newline Snehashis Mukherjee$^1$~~and Sanu Bera$^2$\newline School of Mathematical Sciences, \newline Ramakrishna Mission Vivekananda Educational and Research Institute (rkmveri), \newline Belur Math, Howrah, Box: 711202, West Bengal, India.
 }
\email{\href{mailto:tutunsnehashis@gmail.com}{tutunsnehashis@gmail.com$^1$};\href{mailto:sanubera6575@gmail.com}{sanubera6575@gmail.com$^2$}}


\maketitle

\begin{abstract}
In this article, the quantum spatial ageing algebra due to V. V. Bavula and T. Lu has been studied and a full classification of simple modules are given at root of unity. 
\end{abstract}

\section{\bf{Introduction}}
Let $\mathbb{K}$ be a field and an element $q\in \mathbb{K}$ with $q\neq 0$ and $q^2\neq 1$. The algebra $\mathbb{K}_q[X,Y]:=\mathbb{K}\langle X,Y|XY=qYX\rangle$ is called the quantum plane. 
The quantized enveloping algebra $U_q(\mathfrak{sl}_2)$ of $\mathfrak{sl}_2$ is generated over $\mathbb{K}$ by elements $E,F,K$ and $K^{-1}$ subject to the defining relations:
$$KE=q^2EK,~KF=q^{-2}FK,~EF-FE=\frac{K-K^{-1}}{q-q^{-1}},~ KK^{-1}=K^{-1}K=1.$$
\par The quantum plane and the quantized enveloping algebra $U_q(\mathfrak{sl}_2)$ are important examples of generalized Weyl algebras and ambiskew polynomial rings, see e,g., \cite{vvb} and \cite{dajew}. Let $U_q^{\geq 0}(\mathfrak{sl}_2)$ be the `positive part' of $U_q(\mathfrak{sl}_2)$. It is a subalgebra of $U_q(\mathfrak{sl}_2)$ generated by $K^{\pm 1}$ and $E$. A Hopf algebra structure on $U_q(\mathfrak{sl}_2)$ is defined as follows:
$$\begin{array}{lll}
  \Delta(K)=K\otimes K, & \epsilon(K)=1, & S(K)=K^{-1},\\
  \Delta(E)=E\otimes 1+K\otimes E, & \epsilon(E)=0, & S(E)=-K^{-1}E,\\
  \Delta(F)=F\otimes K^{-1}+1\otimes F, &\epsilon(F)=0,& S(F)=-FK,
\end{array}$$
where $\Delta$ is a comultiplication on $U_q(\mathfrak{sl}_2)$, $\epsilon$ is the counit and $S$ is the antipode of $U_q(\mathfrak{sl}_2)$. Note that the Hopf algebra $U_q(\mathfrak{sl}_2)$ is neither cocommutative nor commutative. There is also a Hopf algebra structure on $U_q^{\geq 0}(\mathfrak{sl}_2)$. 
\par Quantum groups appeared in the $1980$’s and remained a very active area of research. Yet, most of research are related to quantum deformations of semisimple Lie algebras and algebraic groups, there are not many examples of quantum analogue of non-semisimple Lie algebras. One of the interesting examples is the quantum spatial ageing algebra. As mentioned in \cite{bavl}, this algebra is the quantum analogue of the $4$-dimensional (non-semisimple) Lie algebra with basis $\{h,e,x,y\}$ and Lie brackets
\[[h,e]=2e,~[h,x]=x,~[h,y]=-y,~[e,x]=0,~[e,y]=x,~[x,y]=0,\]so-called, the spatial ageing algebra (cf. \cite{luu}). 
\par The notion of smash product has proved to be very useful in studying Hopf algebra actions \cite{sm}. For example, the enveloping algebra of a semi-direct product of Lie algebras can naturally be seen as a smash product algebra. The smash product is constructed from a module algebra, see \cite[4.1]{sm} for details and examples. The quantum plane $\mathbb{K}_q[X,Y]$ is a $U_q(\mathfrak{sl}_2)$-module algebra where
\[\begin{array}{lll}
 K\cdot X=qX,& E\cdot X=0,& F\cdot X=Y\\
  K\cdot Y=q^{-1}Y,& E\cdot Y=X,& F\cdot Y=0
\end{array}\]
and the quantum plane $\mathbb{K}_q[X,Y]$ is also a $U_q^{\geq 0}(\mathfrak{sl}_2)$-module algebra. Then one can form the smash product algebras  \[\mathcal{A}:=\mathbb{K}_q[X,Y]\rtimes U_q^{\geq 0}(\mathfrak{sl}_2)\ \ \text{and}\ \ {A}:=\mathbb{K}_q[X,Y]\rtimes U_q(\mathfrak{sl}_2).\] We call the algebra $\mathcal{A}$ as the quantum spatial ageing algebra. The defining relations for the algebra $\mathcal{A}$ are given explicitly in the following definition.
\begin{defin} \emph{(\cite{bavl})}
The quantum spatial ageing algebra $\mathcal{A}:=\mathbb{K}_q[X,Y]\rtimes U_q^{\geq 0}(\mathfrak{sl}_2)$ is an algebra generated over $\mathbb{K}$ by the elements $E,K,K^{-1},X$ and $Y$ (where $K^{-1}$ is the inverse of $K$) subject to the defining relations:
\begin{equation}\label{e1}
\begin{rcases}
EK=q^{-2}KE,~ XK=q^{-1}KX,~YK=qKY,\\
EX=qXE,~ EY=X+q^{-1}YE, ~ XY=qYX. 
\end{rcases}\end{equation}
\end{defin}
The smash product algebras $\mathcal{A}$ and $A$ are studied in detail in \cite{bavl} and \cite{bavlu} respectively where $q$ is not a root of unity. In particular their prime, primitive and maximal spectra were explicitly described as well as the corresponding factor algebras. The group of automorphisms of $\mathcal{A}$ was determined, and the simple unfaithful $\mathcal{A}$-modules and the simple weight $\mathcal{A}$-modules were classified in \cite{bavl}. Various classes of torsion simple $\mathcal{A}$-modules were classified in \cite{bav1}. The centre of $A$ and a classification of simple weight $A$-modules were obtained in \cite{bavlu}.\\
\textbf{Aim.} In this paper, we continue the study of simple modules over the smash product algebras $\mathcal{A}$ and $A$ at root of unity. Our aim is to construct and classify simple modules over the algebra $\mathcal{A}$ assuming that $\mathbb{K}$ is an algebraically closed field and $q$ is a primitive $l$-th root of unity with $l\geq 3$. Also under the same assumption we wish to classify some class of simple modules for the algebra $A$.\\
\textbf{Arrangement:} The paper is organized as follows: In the first six section we will continue our discussion about the algebra $\mathcal{A}$. In Section $2$, we will discuss about the theory of Polynomial Identity algebras to comment about the $\mathbb{K}$-dimension of the simple modules over $\mathcal{A}$ and $A$. An explicit expression of PI-degree for $\mathcal{A}$ is computed in Section $3$. In Section $4$ and Section $5$ we construct and classify all simple $\mathcal{A}$-modules respectively. In section $6$, we find a class of finite dimensional indecomposable module over $\mathcal{A}$. In Section $7$, we focus on the algebra $A$ along with its PI-deg and give a classification of simple $A$-modules with invertible action of $X$ and $Y$.

\section{\bf{Preliminaries}}
Let $\mathbb{E}$ be the subalgebra of $\mathcal{A}$ generated by the elements $X,E$ and $Y$. The generators of the algebra $\mathbb{E}$ satisfy the defining relations
\[EX=qXE,~YX=q^{-1}XY,~EY-q^{-1}YE=X.\]
The algebra $\mathbb{E}$ has an iterated skew polynomial presentation with respect to the ordering of variables $X,Y,E$ of the form: \[\mathbb{K}[X][Y,\alpha][E,\beta,\delta],\]
where the $\alpha$ and $\beta$ are $\mathbb{K}$-linear automorphisms and the $\delta$ is $\mathbb{K}$-linear $\beta$-derivation such that
\begin{equation}\label{r1}
\alpha(X)=q^{-1}X,~\beta(X)=qX,~\beta(Y)=q^{-1}Y,~\delta(X)=0,~\delta(Y)=X.
\end{equation}
So, the algebra $\mathcal{A}=\mathbb{E}[K^{\pm 1},\tau]$ is a skew Laurent polynomial algebra where $\tau$ is a $\mathbb{K}$-linear automorphism on $\mathbb{E}$ such that
\begin{equation}\label{r2}
    \tau(X)=qX,~\tau(Y)=q^{-1}Y,~\tau (E)=q^{2}E.
\end{equation}
This observation along with the skew polynomial version of the Hilbert Basis Theorem (cf. \cite[Theorem 2.9]{mcr}) yields the following proposition.
\begin{propn}\label{imp}
The algebra $\mathcal{A}$ is an affine noetherian domain. Moreover, the family $\{X^{a}Y^{b}E^{c}K^{\pm d}~|~a,b,c,d\geq 0\}$ is a $\mathbb{K}$-basis of $\mathcal{A}$.
\end{propn}
\par A non zero element $x$ of an algebra ${A}$ is called a normal element if $x{A}={A}x$. Let us consider an element $\phi:=EY-qYE$ of the algebra $\mathcal{A}$ (cf. \cite{bavl}). Then one can easily verify that 
\begin{equation}\label{r3}
    \phi=X+(q^{-1}-q)YE=q^2X+(1-q^2)EY
\end{equation}
\begin{equation}\label{r4}
    X\phi=\phi X,~Y\phi=q\phi Y,~E\phi=q^{-1}\phi E,~K\phi=q\phi K
\end{equation}
So the element $\phi$ is a normal element of the algebra $\mathcal{A}$. Also from the defining relations $(\ref{e1})$ of the algebra $\mathcal{A}$, it follows that $X$ is a normal element of $\mathcal{A}$.
\subsection{Polynomial Identity Algebras}
In this subsection, we recall some known facts concerning Polynomial Identity algebra that we shall be applying on quantum spatial ageing algebra at root of unity for further development. First we recall a result which provides a sufficient condition for a ring to be PI. 

\begin{prop}\emph{(\cite[Corollary 13.1.13]{mcr})}\label{f}
If $R$ is a ring which is finitely generated module over a commutative subring, then $R$ is a PI ring.
\end{prop}

\begin{lemm}\label{c1}
For $i\geq 1$, the following identities hold in the algebra $\mathcal{A}$:
\begin{enumerate}
    \item [(i)] $EY^{i}=q^{-i}Y^iE+\displaystyle\frac{q^{-2i}-1}{q^{-2}-1}XY^{i-1}$
    \item [(ii)] $YE^{k}=q^{i}E^{i}Y-\displaystyle\frac{q(1-q^{2i})}{1-q^2}XE^{i-1}.$
\end{enumerate}
The equalities can be proved by induction on $i$ and using the relation $EY=X+q^{-1}YE$. With this we have:
\end{lemm}
\begin{lemm}\label{cc1}
If $q$ is a primitive $l$-th root of unity ($l\geq 3$), then $K^{\pm l},E^{l},X^{l}$ and $Y^{l}$ are contained in the center of $\mathcal{A}$.
\end{lemm}
These two results will be used throughout this article.
\begin{prop} \label{finite}
The algebra $\mathcal{A}$ is a PI algebra if and only if $q$ is a root of unity.
\end{prop}
\begin{proof}
Suppose $q$ be a primitive $l$-th root of unity. Let $Z$ be the subalgebra of $\mathcal{A}$ generated by $K^{\pm l},E^l,X^l$ and $Y^l$. Then by Lemma $\ref{cc1}$, the subalgebra $Z$ is central. Now one can easily verify that $\mathcal{A}$ is a finitely generated module over this central subalgebra $Z$, with basis $$\{X^{a}Y^{b}E^{c}K^{\pm d}~|~0\leq a,b,c,d< l\}.$$ Hence it follows from Proposition \ref{f} that $\mathcal{A}$ is PI algebra.
\par For the converse just note that the $\mathbb{K}$-subalgebra of $\mathcal{A}$ generated by $X$ and $Y$ with relation $XY=qYX$ is not PI if $q$ is not a root of unity. (cf. \cite[Proposition I.14.2.]{brg})
\end{proof}
\begin{rema}\label{re1}
Now the quantum spatial ageing algebra $\mathcal{A}$ being finite module over a central subalgebra, it follows from a standard result (cf. \cite[Proposition III.1.1.]{brg}) that every simple $\mathcal{A}$-module is finite dimensional vector space over $\mathbb{K}$.
\end{rema}
\par Primitive PI ring exhibit a particularly nice structure established in Kaplansky's Theorem (cf. \cite[Theorem 13.3.8]{mcr}).  Now Kaplansky's Theorem has a striking consequence in case of a prime affine PI algebra over an algebraically closed field.
\begin{prop}\emph{(\cite[Theorem I.13.5]{brg})}\label{sim}
Let $A$ be a prime affine PI algebra over an algebraically closed field $\mathbb{K}$, with PI-deg($A$) = $n$ and $V$ be an simple $A$-module. Then $V$ is a vector space over $\mathbb{K}$ of dimension $t$, where $t \leq n$, and $A/ann_A(V) \cong M_t(\mathbb{K})$.
\end{prop}
The above proposition yields the important link between the PI degree of a prime affine PI algebra over an algebraically closed field and its irreducible representations. Thus from Proposition $\ref{finite}$ and Proposition $\ref{imp}$ along with Proposition $\ref{sim}$, it is quite clear that each simple $\mathcal{A}$-module is finite dimensional and can have dimension at most $\pideg( \mathcal{A})$. In the next subsection we shall discuss about the PI degree of quantum affine space which help us to obtain an explicit expression of PI degree of $\mathcal{A}$ in Section $3$.
\subsection{PI Degree of Quantum Affine Spaces} The following result of De Concini and Procesi provides one of the key techniques for calculating the PI degree of a quantum affine space.
\begin{prop}\emph{(\cite[Proposition 7.1]{di})}\label{quan}
Let $\mathbf{q}=\left(q_{ij}\right)$ be an  $n \times n$ multiplicatively antisymmetric matrix over $\mathbb{K}$.
 Suppose that $q_{ij}=q^{h_{ij}}$ for all $i,j$, where $q \in \mathbb{K}^*$ is a primitive $m$-th root of unity and the $h_{ij} \in \mathbb{Z}$. Let $h$ be the cardinality of the image of the homomorphism 
\[
    \mathbb{Z}^n \xrightarrow{(h_{ij})} \mathbb{Z}^n \xrightarrow{\pi} \left(\mathbb{Z}/m\mathbb{Z}\right)^n,
\]
where $\pi$ denotes the canonical epimorphism. Then \[\pideg(\mathcal{O}_{\mathbf{q}}(\mathbb{K}^n))=\pideg(\mathcal{O}_{\mathbf{q}}\left(\mathbb{(K^*)}^n\right)=\sqrt{h}.\]
\end{prop}
It is well known that a skew symmetric matrix over $\mathbb{Z}$ such as our matrix $H:=(h_{ij})$ can be brought into a $2\times 2$ block diagonal form (commonly known as skew normal form): 
\[\diagonal\left(\begin{pmatrix}
0&h_1\\
-h_1&0
\end{pmatrix},\cdots,\begin{pmatrix}
0&h_s\\
-h_s&0
\end{pmatrix},\bf{0}\right),\]
where $\bf{0}$ is the square matrix of zeros of dimension equals $\dime(\kere H)$, so that $2s=\ran(H)=n-\dime(\kere H)$ and $h_i\mid h_{i+1}\in \mathbb{Z}\setminus\{0\},~~\forall \ \ 1\leq i \leq s-1$. This non zero $h_1,h_1,\cdots,h_s,h_s$ are called the invariant factors of $H$. The following result simplifies the calculation of $h$ in the statement of Proposition \ref{quan} by the properties of the integral matrix $H$, namely dimension of its kernel, along with its invariant factors and the value of $m$.
\begin{lemm}\emph{(\cite[Lemma 5.7]{ar})}\label{mainpi}
Take $1\neq q\in \mathbb{K}^*$, a primitive $m$-th root of unity. Let $H$ be a skew symmetric integral matrix associated to $\mathbf{q}$ with invariant factors $h_1,h_1,\cdots,h_s,h_s$. Then PI degree of $\mathcal{O}_{\mathbf{q}}(\mathbb{K}^n)$ is given as \[\pideg(\mathcal{O}_{\mathbf{q}}(\mathbb{K}^n))=\pideg(\mathcal{O}_{\mathbf{q}}\left(\mathbb{(K^*)}^n\right)=\prod_{i=1}^{\frac{n-\dime(\kere H)}{2}}\frac{m}{\gcdi(h_i,m)}.\]
\end{lemm}
In the next section we will focus on computing the PI degree of quantum spatial ageing algebra. 

\section{\bf{PI degree for $\mathcal{A}$}}
In this section we wish to calculate the PI-degree of $\mathcal{A}$ with the help of Proposition \ref{quan} and Lemma \ref{mainpi} along with a result on derivation erasing in \cite[Theorem 7]{lm2}.
\par First recall the subalgebra $\mathbb{E}$ generated by $X,Y,E$. Let $\mathcal{B}$ be the subalgebra of $\mathcal{A}$ generated by the $K$ and the elements of $\mathbb{E}$. Then $\mathcal{B}$ is an iterated skew polynomial ring with respect to the ordering of variables $X,Y,E,K$ of the form: 
\[\mathcal{B}=\mathbb{K}[X][Y,\alpha][E,\beta,\delta][K,\tau]=\mathbb{E}[K,\tau],\]
where the $\mathbb{K}$-linear maps $\alpha,\beta,\tau$ and $\delta$ are as defined in $(\ref{r1})$ and $(\ref{r2})$. Observe that $\mathcal{B}\subset \mathcal{A}\subset \cf\mathcal{B}$ and hence $\pideg \mathcal{A}=\pideg \mathcal{B}$ (cf. \cite[Corollary I.13.3]{brg}). Note that the skew relation $\beta \delta=q^2\delta \beta,~(q^2\neq 1)$ is satisfied on $\mathbb{K}[X][Y,\alpha]$. Then the derivation erasing process (independent of characteristic) in \cite[Theorem 7]{lm2} provides $\cf {\mathcal{B}}\cong \cf \mathcal{O}_{\mathbf{q}}(\mathbb{K}^{4})$, where the $4\times 4$ matrix of relations $\mathbf{q}$ is 
\[\mathbf{q}=\begin{pmatrix}
1&q&q^{-1}&q^{-1}\\
q^{-1}&1&q&q\\
q&q^{-1}&1&q^{-2}\\
q&q^{-1}&q^{2}&1
\end{pmatrix}.\]
The power of $q$ from the matrix $\mathbf{q}$ give a $4\times 4$ integral matrix 
\[\mathbf{q}'=\begin{pmatrix}
0&1&-1&-1\\
-1&0&1&1\\
1&-1&0&-2\\
1&-1&2&0
\end{pmatrix}.\]
Then $\pideg \mathcal{A}$ can be computed using the integral matrix $\mathbf{q}'$ in the Proposition \ref{quan} with the help of Lemma \ref{mainpi}.  The cardinality of the image will not be changed if we first perform some elementary reductions on $\mathbf{q}'$. Now we can apply following elementary operations on $\mathbf{q}'$ to reduce it into a simpler form:
\begin{itemize}
    \item Replace $\rowa~(1)$ with $\rowa~(1)+\rowa~(2)$ and replace $\colu~(1)$ with $\colu~(1) +\colu~(2)$.
    \item Replace $\rowa~(3)$ with $\rowa~(3)+\rowa~(1)$ and replace $\colu~(3)$ with $\colu~(3)+ \colu~(1)$.
    \item Replace $\rowa~(4)$ with $\rowa~(4)+\rowa~(1)$ and replace $\colu~(4)$ with $\colu~(4)+ \colu~(1)$.
\end{itemize}
Then we have $4\times 4$ integral matrix $\mathbf{q}^{''}$ of the form 
\[\mathbf{q}^{''}=\begin{pmatrix}
0&1&0&0\\
-1&0&0&0\\
0&0&0&-2\\
0&0&2&0
\end{pmatrix}.\]
Thus the integral matrices $\mathbf{q}^{'}$ and $\mathbf{q}^{''}$ are congurent and hence they have the same invariant factors $1,1,2,2$. Therefore from above discussions along with the help of Lemma \ref{mainpi} we have 
$$\text{PI deg}~\mathcal{A}=\text{PI deg}~\mathcal{B}=\text{PI deg}~\mathcal{O}_{\mathbf{q}}(\mathbb{K}^4)=
\begin{cases}
 l^2,&  l~ \text{odd}\\
 \frac{l^2}{2}, &  l~ \text{even}
\end{cases}$$
\begin{propn}
Let $q$ be a primitive $l$-th root of unity. Then $\pideg \mathcal{A}=
\begin{cases}
 l^2,&  l~ \text{odd}\\
 \frac{l^2}{2}, &  l~ \text{even}
 \end{cases}$
\end{propn}

\section{\bf{Construction of Simple $\mathcal{A}$-Modules}}
In this section we wish to construct simple $\mathcal{A}$-modules depending on some scalar parameters. Assume that $q$ be a primitive $l$-th root of unity ($l\geq 3$). Also we let 
\[l_1:=\ord (q^2)=\begin{cases}
     l,& l~\text{odd}\\
     \frac{l}{2},& l~\text{even.}
     \end{cases}\]
\subsection{Simple modules of type $M_1(\underline{\mu})$} For $\underline{\mu}:=(\mu_1,\mu_2,\mu_3,\mu_4)\in \mathbb{({K}^*)}^4$, let us consider the $\mathbb{K}$-vector space $M_1(\underline{\mu})$ with basis
\[\{e(a_1,a_2)~|~0\leq a_1 \leq l_1-1,~0\leq a_2 \leq l-1\}.\]
Define an ${\mathcal{A}}$-module structure on the $\mathbb{K}$-space $M_1(\underline{\mu})$ as follows: 
\[
 e(a_1,a_2)X=\mu_1q^{a_1+a_2}e(a_1,a_2),\ \
 e(a_1,a_2)K^{\pm 1}= \mu_4^{\pm 1}e(a_1,a_2+(\pm 1)).\]
 When $l$ is odd, the action of $E$ and $Y$ is given as follows
 \[\begin{array}{l}
 e(a_1,a_2)E= \mu_3q^{2a_2}e(a_1+1,a_2),\\
 e(a_1,a_2)Y=\mu_3^{-1}q^{-(a_1+a_2)}\left(\displaystyle\frac{q\mu_2-q^{2a_1+1}\mu_1}{1-q^2}\right)e(a_1+(-1),a_2).\end{array}\]
When $l$ is even, the action of $E$ and $Y$ is defined as follows
\[\begin{array}{l}
 e(a_1,a_2)E=\begin{cases}
  \mu_3q^{2a_2}e(a_1\dotplus 1,a_2),& a_1\neq \frac{l}{2}-1\\
    \mu_3q^{2a_2}e\left(0,\frac{l}{2}+a_2\right),& a_1=\frac{l}{2}-1
 \end{cases} \\
e(a_1,a_2)Y=\begin{cases}
  \mu_3^{-1}q^{-(a_1+a_2)}\left(\displaystyle\frac{q\mu_2-q^{2a_1+1}\mu_1}{1-q^2}\right)e(a_1\dotplus (-1),a_2),& a_1\neq 0\\
    \mu_3^{-1}q^{-a_2}\left(\displaystyle\frac{q(\mu_2-\mu_1)}{1-q^2}\right)e\left(\frac{l}{2}-1,\frac{l}{2}+a_2\right),& a_1=0.\\
\end{cases}
 \end{array}\]
where $+$ and $\dotplus$ are addition modulo $l$ and $\frac{l}{2}$ respectively. In order to establish the well-definedness we need to check that the $\mathbb{K}$-endomorphisms of $M_1(\underline{\mu})$ defined by the above rules satisfy the relations $(\ref{e1})$. Indeed when $l$ is odd:
\[\begin{array}{ll}
e(a_1,a_2)EK&=\mu_3q^{2a_2}e(a_1+1,a_2)K\\
&=\mu_3\mu_4q^{2a_2}e(a_1+1,a_2+1)=q^{-2}e(a_1,a_2)KE,\\
e(a_1,a_2)YE&=\mu_3^{-1}q^{-(a_1+a_2)}\left(\displaystyle\frac{q\mu_2-q^{2a_1+1}\mu_1}{1-q^2}\right)e(a_1+(-1),a_2)E\\
&=q^{-a_1+a_2}\left(\displaystyle\frac{q\mu_2-q^{2a_1+1}\mu_1}{1-q^2}\right)e(a_1,a_2),\\
e(a_1,a_2)EY&=\mu_3q^{2a_2}e(a_1+1,a_2)Y\\
&=q^{-1-a_1+a_2}\left(\displaystyle\frac{q\mu_2-q^{2a_1+3}\mu_1}{1-q^2}\right)e(a_1,a_2)\\
&=e(a_1,a_2)X+q^{-1}e(a_1,a_2)YE.
\end{array}\]
Similarly the remaining relations can be easily verified. When $l$ is even:
\[\begin{array}{ll}
e\left(\frac{l}{2}-1,a_2\right)EK &=\mu_3q^{2a_2}e(0,\frac{l}{2}+a_2)K\\
&=\mu_3\mu_4q^{2a_2}e\left(0,\frac{l}{2}+a_2+1\right)=q^{-2}e\left(\frac{l}{2}-1,a_2\right)KE,\\
e\left(\frac{l}{2}-1,a_2\right)EY&=\mu_3q^{2a_2}e\left(0,\frac{l}{2}+a_2\right)Y\\
&=q^{-(\frac{l}{2}+a_2)}\left(\displaystyle\frac{q(\mu_2-\mu_1)}{1-q^2}\right)e\left(\frac{l}{2}-1,a_2\right)\\
&=e\left(\frac{l}{2}-1,a_2\right)X+q^{-1}e\left(\frac{l}{2}-1,a_2\right)YE,\\
e\left(0,a_2\right)YE&=\mu_3^{-1}q^{-a_2}\left(\displaystyle\frac{q(\mu_2-\mu_1)}{1-q^2}\right)e\left(\frac{l}{2}-1,\frac{l}{2}+a_2\right)E\\
&=q^{a_2}\left(\displaystyle\frac{q(\mu_2-\mu_1)}{1-q^2}\right)e\left(0,a_2\right),\\
e\left(0,a_2\right)EY&=q^{-1+a_2}\left(\displaystyle\frac{q\mu_2-q^{3}\mu_1}{1-q^2}\right)e(a_1,a_2)\\
&=e(0,a_2)X+q^{-1}e(0,a_2)YE.
\end{array}\] All the remaining relations can be easily verified. With this we have 
\begin{theo}\label{f1}
The module $M_1(\underline{\mu})$ is a simple ${\mathcal{A}}$-module of dimension $l_1l$.
\end{theo}
\begin{proof}
Let $P$ be a non-zero submodule of $M_1(\underline{\mu})$. We claim that $P$ contains a basis vector of the form $e(a_1,a_2)$. Indeed, any member $p\in P$ is a finite $\mathbb{K}$-linear combination of such vectors. i.e.,
\[
  p:=\sum_{\text{finite}} \lambda_k~e\left(a_1^{(k)},a_2^{(k)}\right)  
\]
for some $\lambda_k\in \mathbb{K}$. Suppose there exist two non-zero coefficients, say, $\lambda_u,\lambda_v$. Then in particular concentrate on the two distinct vectors $e\left(a_1^{(u)},a_2^{(u)}\right)$ and $e\left(a_1^{(v)},a_2^{(v)}\right)$.\\
\textbf{Case 1:} First consider 
\begin{equation}\label{yo}
    a_1^{(u)}+ a_2^{(u)}\not\equiv a_1^{(v)}+ a_2^{(v)}\ \  (\text{mod}\  l).
\end{equation} Then the vectors $e\left(a_1^{(u)},a_2^{(u)}\right)$ and $e\left(a_1^{(v)},a_2^{(v)}\right)$ are eigenvectors of $X$ associated with the eigenvalues
$\Lambda_u=\mu_1q^{a_1^{(u)}+a_2^{(u)}}$ and $\Lambda_v=\mu_1q^{a_1^{(v)}+a_2^{(v)}}$ respectively. Infact it follows from $(\ref{yo})$ that $\Lambda_u$ and $\Lambda_v$ are distinct.\\ 
\textbf{Case 2:} Next consider \begin{equation}\label{co}
a_1^{(u)}+ a_2^{(u)}\equiv a_1^{(v)}+ a_2^{(v)}~ (\text{mod}~ l),
\end{equation}
then the vectors $e\left(a_1^{(u)},a_2^{(u)}\right)$ and $e\left(a_1^{(v)},a_2^{(v)}\right)$ are eigenvectors of $\phi~(=EY-qYE)$ associated with the eigenvalues $\Lambda'_u=\mu_2q^{-a_1^{(u)}+a_2^{(u)}}$ and $\Lambda'_v=\mu_2q^{-a_1^{(v)}+a_2^{(v)}}$ respectively. We claim that $\Lambda'_u \neq \Lambda'_v$. Indeed,
\begin{equation}\label{coo}
    \Lambda'_u=\Lambda'_v\implies {-a_1^{(u)}+a_2^{(u)}}\equiv {-a_1^{(v)}+a_2^{(v)}}~(\text{mod}~l),
\end{equation}
and then form $(\ref{co})$ and $(\ref{coo})$ we obtain
\begin{equation}\label{to}
    2a_1^{(u)}\equiv 2a_1^{(v)}~(\text{mod}~l),\ \ 2a_2^{(u)}\equiv 2a_2^{(v)}~(\text{mod}~l).
\end{equation}
If $l$ is odd, the equation $(\ref{to})$ implies $e\left(a_1^{(u)},a_2^{(u)}\right)=e\left(a_1^{(v)},a_2^{(v)}\right)$, which is a contradiction. Also if $l$ is even, the equations $(\ref{to})$ and $(\ref{co})$ together implies
\[a_1^{(u)}\equiv a_1^{(v)}~(\text{mod}~\frac{l}{2}),\ \ a_2^{(u)}\equiv a_2^{(v)}~(\text{mod}~l)\] and hence $e\left(a_1^{(u)},a_2^{(u)}\right)=e\left(a_1^{(v)},a_2^{(v)}\right)$, which is a contradiction.
\par Now in either case $pX-\Lambda_up$ or $p\phi-\Lambda'_up$ is a non zero element in $P$ of smaller length than $p$. Hence by induction it follows that every non zero submodule of $M_1(\underline{\mu})$ contains a basis vector of the form $e(a_1,a_2)$. Thus $M_1(\underline{\mu})$ is simple $\mathcal{A}$-module by the action of $E,Y$ and $K$.
\end{proof}
\begin{prop}\label{pri}
For $\underline{\mu},\underline{\gamma}\in(\mathbb{K}^*)^4$, the simple $\mathcal{A}$-modules $M_1(\underline{\mu})$ and $M_1(\underline{\gamma})$ are isomorphic if and only if  $\underline{\mu}$ and $\underline{\gamma}$ are related by 
\begin{equation}\label{qiso}
\mu_1=\gamma_1q^{r_1+r_2},\ \mu_2=\gamma_2q^{-r_1+r_2},\ \
\mu^l_3=\gamma^l_3,\ \  \mu^l_4=\gamma^l_4,
\end{equation} for some  $r_1,r_2$ with {$0\leq r_1\leq l_1-1,~0\leq r_2\leq l-1$.}
\end{prop}
\begin{proof}
Let $\Psi:M_1(\underline{\mu})\rightarrow M_1(\underline{\gamma})$ be an $\mathcal{A}$-module isomorphism. As in the space $M_1(\underline{\mu})$
\[e(a_1,a_2)=\mu_3^{-a_1}\mu_4^{-a_2}e(0,0)E^{a_1}K^{a_2},\]
then $\Psi$ can be uniquely determined by the only image of $e(0,0)$, i.e., say
\begin{equation}\label{sum}
\Psi\left(e(0,0)\right)= \sum_{\text{finite}}\lambda_t~e\left(b^{(t)}_1,b^{(t)}_2\right)   
\end{equation}
 for $\lambda_t \in \mathbb{K}^*$. Suppose there exist two non-zero coefficients, say, $\lambda_u,\lambda_v$. Now equating the coefficient of basis vectors on both sides of the equalities \begin{equation}\label{rec}
 \Psi\left(e(0,0)X\right)=\Psi\left(e(0,0)\right)X\ \  \text{and}\ \  \Psi\left(e(0,0)\phi\right)=\Psi\left(e(0,0)\right)\phi
\end{equation} we must have, respectively,
\[\mu_{1}=\gamma_{1}q^{b^{(u)}_1+b^{(u)}_2}=\gamma_{1}q^{b^{(v)}_1+b^{(v)}_2}
\ \ \text{and}\ \ \mu_{2}=\gamma_{2}q^{-b^{(u)}_1+b^{(u)}_2}=\gamma_{2}q^{-b^{(v)}_1+b^{(v)}_2}.\] 
This implies 
\[b^{(u)}_1+b_2^{(u)}\equiv b^{(v)}_1+b_2^{(v)}(\text{mod}~l)\ \ \text{and}\ \  -b^{(u)}_1+b_2^{(u)}\equiv -b^{(v)}_1+b_2^{(v)}(\text{mod}~l)\]
and finally we get $b^{(u)}_1\equiv b_1^{(v)}~(\text{mod}~l_1)$ and $b^{(u)}_2\equiv b_2^{(v)}~(\text{mod}~l)$, which is a contradiction. Hence it follows that the image in  (\ref{sum}) is of the form \[\Psi\left(e(0,0)\right)= \lambda~e\left(r_1,r_2\right),\] for some $\lambda \in \mathbb{K}^*$ and for some $r_1,r_2$ with {$0\leq r_1\leq l_1-1,~0\leq r_2\leq l-1$.} Also using this form of $\Psi$ along with the equalities $(\ref{rec})$ and 
\[\Psi\left(e(0,0)E^l\right)=\Psi\left(e(0,0)\right)E^l,\ \ \Psi\left(e(0,0)K^l\right)=\Psi\left(e(0,0)\right)K^l,\]
we can obtain the required relation $(\ref{qiso})$ between $\underline{\mu}$ and $\underline{\gamma}$.
\par For the converse part define a $\mathbb{K}$-linear map $\Phi:M_1(\underline{\mu})\rightarrow M_1(\underline{\gamma})$ by setting, when {$l$ is odd:}
\[\Phi \left(e(a_1,a_2)\right):=\left(\mu_3^{-1}\gamma_3q^{2r_2}\right)^{a_1}\left(\mu_4^{-1}\gamma_4\right)^{a_2}e\left(a_1+r_1,a_2+r_2\right);\]
when {$l$ is even:} 
\[\begin{array}{l}
\Phi\left(e(a_1,a_2)\right):=\\
\begin{cases}
 \left(\mu_3^{-1}\gamma_3q^{2r_2}\right)^{a_1}\left(\mu_4^{-1}\gamma_4\right)^{a_2}e\left(a_1\dotplus r_1,a_2+r_2\right),&\ 0\leq a_1\leq \frac{l}{2}-r_1-1\\
 \left(\mu_3^{-1}\gamma_3q^{2r_2}\right)^{a_1}\left(\mu_4^{-1}\gamma_4\right)^{a_2}e\left(a_1\dotplus r_1,\frac{l}{2}+a_2+r_2\right),&\ \frac{l}{2}-r_1-1<a_1\leq \frac{l}{2}-1
\end{cases}
\end{array}\]
where $+$ and $\dotplus$ are addition modulo $l$ and $\frac{l}{2}$ respectively. It is easy to verify using the relations $(\ref{qiso})$ that $\Phi$ is an $\mathcal{A}$-module isomorphism.
\end{proof}

\subsection{Simple modules of type $M_2(\underline{\mu})$} For $\underline{\mu}:=(\mu_1,\mu_2,\mu_3)\in \mathbb{({K}^*)}^3$, let us consider the $\mathbb{K}$-vector space $M_2(\underline{\mu})$ with basis
\[\{e(a_1,a_2)~|~0\leq a_1 \leq l_1-1,~0\leq a_2 \leq l-1\}.\] 
Define an action of ${\mathcal{A}}$ on the $\mathbb{K}$-space $M_2(\underline{\mu})$ as follows:
\[ e(a_1,a_2)X=\mu_1q^{-a_1+a_2}e(a_1,a_2),\ \
 e(a_1,a_2)K^{\pm 1}= \mu_3^{\pm 1}e(a_1,a_2+(\pm 1)).\]
When $l$ is odd, the action of $E$ and $Y$ is given as follows
 \[\begin{array}{l}
 e(a_1,a_2)E=\begin{cases}
 0,&a_1=0\\
-\mu_2^{-1}\mu_1q^{a_1+2a_2}\displaystyle\frac{q^{-2a_1}-1}{q^{-2}-1}e(a_1+(-1),a_2),&a_1\neq 0
\end{cases}\\
 e(a_1,a_2)Y=\mu_2q^{-a_2}e(a_1+1,a_2).\end{array}\]
When $l$ is even, the action of $E$ and $Y$ is defined as follows
\[\begin{array}{l}
 e(a_1,a_2)E=\begin{cases}
  0,& a_1=0\\
  -\mu_2^{-1}\mu_1q^{a_1+2a_2}\displaystyle\frac{q^{-2a_1}-1}{q^{-2}-1}e(a_1\dotplus (-1),a_2),& a_1\neq 0
 \end{cases} \\
e(a_1,a_2)Y=\begin{cases}
\mu_2q^{-a_2}e(a_1\dotplus 1,a_2),& a_1\neq \frac{l}{2}-1\\
    \mu_2q^{-a_2}e(0,\frac{l}{2}+a_2),& a_1=\frac{l}{2}-1.\\
\end{cases}
 \end{array}\]
where $+$ and $\dotplus$ are addition modulo $l$ and $\frac{l}{2}$ respectively. As in the $M_1(\underline{\mu})$, it can be easily checked that the $\mathbb{K}$-endomorphisms of $M_2(\underline{\mu})$ defined by the above rules satisfy the relations $(\ref{e1})$.
 \begin{theo}\label{f2}
The module $M_2(\underline{\mu})$ is a simple ${\mathcal{A}}$-module of dimension $l_1l$.
\end{theo}
\begin{proof}
Note that the vector $e(a_1,a_2)\in M_2(\underline{\mu})$ is an eigenvector of $X$ and $\phi~(=EY-qYE)$ associated with the eigenvalue $\mu_1q^{-a_1+a_2}$ and $\mu_1q^{a_1+a_2+2}$ respectively. With this fact the proof is parallel to the proof of Theorem $\ref{f1}$.  
\end{proof}

\begin{prop}
For $\underline{\mu},\underline{\gamma}\in(\mathbb{K}^*)^3$, the simple $\mathcal{A}$-modules $M_2(\underline{\mu})$ and $M_2(\underline{\gamma})$ are isomorphic if and only if  $\underline{\mu}$ and $\underline{\gamma}$ are related by 
\begin{equation}\label{m2}
\mu_1=\gamma_1q^{-r_1+r_2},\ \ \mu^l_2=\gamma^l_2,\ \  \mu^l_3=\gamma^l_3,
\end{equation} for some  $r_1,r_2$ with {$0\leq r_1\leq l_1-1,~0\leq r_2\leq l-1$.}
\end{prop}
\begin{proof}
As in the space $M_2(\underline{\mu})$
\[e(a_1,a_2)=\mu_2^{-a_1}\mu_3^{-a_2}e(0,0)Y^{a_1}K^{a_2},\]
then $\mathcal{A}$-module isomorphism $\Psi:M_2(\underline{\mu})\rightarrow M_2(\underline{\gamma})$ can be uniquely determined by the only image of $e(0,0)$. Now repeating the argument of Proposition $\ref{pri}$ we have \[\Psi\left(e(0,0)\right)= \lambda~e\left(r_1,r_2\right),\] for some $\lambda \in \mathbb{K}^*$ and for some $r_1,r_2$ with {$0\leq r_1\leq l_1-1,~0\leq r_2\leq l-1$.} Hence the required relations $(\ref{m2})$ follows from the equalities  \[\Psi\left(e(0,0)X\right)=\Psi\left(e(0,0)\right)X,\ \ \Psi\left(e(0,0)Y^l\right)=\Psi\left(e(0,0)\right)Y^l\]
\[\text{and}\ \ \Psi\left(e(0,0)K^l\right)=\Psi\left(e(0,0)\right)K^l.\]
\par For the converse part define a $\mathbb{K}$-linear map $\Phi:M_2(\underline{\mu})\rightarrow M_2(\underline{\gamma})$ by setting, when {$l$ is odd:}
\[\Phi \left(e(a_1,a_2)\right):=\left(\mu_2^{-1}\gamma_2q^{-r_2}\right)^{a_1}\left(\mu_3^{-1}\gamma_3\right)^{a_2}e\left(a_1+r_1,a_2+r_2\right);\]
when {$l$ is even:} 
\[\begin{array}{l}
\Phi\left(e(a_1,a_2)\right):=\\
\small{\begin{cases}
 \left(\mu_2^{-1}\gamma_2q^{-r_2}\right)^{a_1}\left(\mu_3^{-1}\gamma_3\right)^{a_2}e\left(a_1\dotplus r_1,a_2+r_2\right),&0\leq a_1\leq\frac{l}{2}-r_1-1\\
q^{-\frac{l}{2}(a_1+r_1-\frac{l}{2})}\left(\mu_2^{-1}\gamma_2q^{-r_2}\right)^{a_1}\left(\mu_3^{-1}\gamma_3\right)^{a_2}e\left(a_1\dotplus r_1,\frac{l}{2}+a_2+r_2\right),&\frac{l}{2}-r_1-1<a_1\leq\frac{l}{2}-1
\end{cases}}
\end{array}\] where $+$ and $\dotplus$ are addition modulo $l$ and $\frac{l}{2}$ respectively. It is easy to verify using the relations in $(\ref{m2})$ that $\Phi$ is an $\mathcal{A}$-module isomorphism.
\end{proof}

\subsection{Simple modules of type $M_3(\underline{\mu})$} For $\underline{\mu}:=(\mu_1,\mu_2)\in \mathbb{({K}^*)}^2$, let us consider the $\mathbb{K}$-vector space $M_3(\underline{\mu})$ with basis consisting of 
\[\{e(a_1,a_2)~|~0\leq a_1 \leq l_1-1,~0\leq a_2 \leq l-1\}.\]
Define an action of ${\mathcal{A}}$ on the $\mathbb{K}$-space $M_3(\underline{\mu})$ by 
\[\begin{array}{l}
e(a_1,a_2)X=\mu_1q^{-a_1+a_2}e(a_1,a_2),\\
e(a_1,a_2)K^{\pm 1}= \mu_2^{\pm 1}e(a_1,a_2+(\pm 1)),\\
e(a_1,a_2)E=\begin{cases}
-q^{a_1+2a_2}\mu_1\displaystyle\frac{q^{-2a_1}-1}{q^{-2}-1}e(a_1\oplus (-1),a_2),\hspace{.6cm} a_1\neq 0\\
 0,\hspace{.6cm} a_1=0
\end{cases}\\
e(a_1,a_2)Y=\begin{cases}
q^{-a_2}e(a_1\oplus 1,a_2), \hspace{.6cm} a_1\neq l_1-1\\
  0, \hspace{.6cm} a_1= l_1-1
\end{cases}
\end{array}\]
where $+$ and $\oplus$ are addition in the additive group $\mathbb{Z}/l\mathbb{Z}$ and $\mathbb{Z}/{l_1\mathbb{Z}}$ respectively. As in the $M_1(\underline{\mu})$, it can be easily verified that the above action satisfy the defining relations $(\ref{e1})$ of $\mathcal{A}$.
 \begin{theo}\label{f3}
The module $M_3(\underline{\mu})$ is a simple ${\mathcal{A}}$-module of dimension $l_1l$.
\end{theo}
\begin{proof}
Note that the vector $e(a_1,a_2)\in M_3(\underline{\mu})$ is an eigenvector of $X$ and $\phi~(=EY-qYE)$ associated with the eigenvalue $\mu_1q^{-a_1+a_2}$ and $\mu_1q^{a_1+a_2+2}$ respectively. Using this fact the proof is analogous to the proof of Theorem $\ref{f1}$. 
\end{proof}
\begin{prop}
For $\underline{\mu},\underline{\gamma}\in(\mathbb{K}^*)^2$, the simple $\mathcal{A}$-modules $M_3(\underline{\mu})$ and $M_3(\underline{\gamma})$ are isomorphic if and only if  $\underline{\mu}$ and $\underline{\gamma}$ are related by 
\begin{equation}\label{m3}
\mu_1=\gamma_1q^{r},\ \mu^l_2=\gamma^l_2,
\end{equation} for some  $r$ with $0\leq r\leq l-1$.
\end{prop}
\begin{proof}
As in the space $M_3(\underline{\mu})$
\[e(a_1,a_2)=\mu_2^{-a_2}e(0,0)Y^{a_1}K^{a_2},\]
then $\mathcal{A}$-module isomorphism $\Psi:M_3(\underline{\mu})\rightarrow M_3(\underline{\gamma})$ can be uniquely determined by the only image of $e(0,0)$. Now repeating the argument of Proposition $\ref{pri}$ we have \[\Psi\left(e(0,0)\right)= \lambda~e\left(b_1,b_2\right),\] for some $\lambda \in \mathbb{K}^*$ and for some $b_1,b_2$ with {$0\leq b_1\leq l_1-1,~0\leq b_2\leq l-1$.} We claim that $b_1=0$. Indeed
\[\Psi\left(e(0,0)Y^{l_1-b_1}\right)=\lambda~e\left(b_1,b_2\right)Y^{l_1-b_1}=0
\implies e(0,0)Y^{l_1-b_1}=0\implies b_1=0.\]
Hence the required relations $(\ref{m3})$ follows from the equalities  \[\Psi\left(e(0,0)X\right)=\Psi\left(e(0,0)\right)X\ \ \text{and}\ \  \Psi\left(e(0,0)K^l\right)=\Psi\left(e(0,0)\right)K^l.\]
\par For the converse part define a $\mathbb{K}$-linear map $\Phi:M_3(\underline{\mu})\rightarrow M_3(\underline{\gamma})$ by setting
\[\Phi \left(e(a_1,a_2)\right):=q^{-ra_1}\left(\mu_2^{-1}\gamma_2\right)^{a_2}e\left(a_1,a_2+r\right).\]
It is easy to verify using the relations in $(\ref{m3})$ that $\Phi$ is an $\mathcal{A}$-module isomorphism.
\end{proof}
\begin{remak}
It is clear from the action of $\mathcal{A}$ that there does not exist any isomorphism between different types of simple $\mathcal{A}$-modules. Indeed 
\begin{itemize}
    \item no non zero element of $M_1(\underline{\mu})$ is annihilated by $E$, but $e(0,a_2)$ in $M_2(\underline{\mu})$ or $M_3(\underline{\mu})$ is annihilated by $E$ and 
    \item no non zero element of $M_2(\underline{\mu})$ annihilated by $Y$, but $e(l_1-1,a_2)$ in $M_3(\underline{\mu})$ is annihilated by $Y$.
\end{itemize} 
\end{remak}
\section{\bf{Classification of simple ${\mathcal{A}}$-modules}}
Let $q$ be a primitive $l$-th root of unity and $N$ be a simple module over ${\mathcal{A}}$. Then by Remark $\ref{re1}$ $N$ is a finite dimensional vector space over $\mathbb{K}$. If $x\in \mathcal{A}$ is a normal element such that the operator $x$ has an eigen vector on $N$ corresponding to the eigen value $0$, then $\ke(x):=\{u\in N~|~ux=0\}$ is a non zero submodule of $N$ and hence $\ke(x)=N$. Thus in such case we call the action of $x$ on $N$ is trivial, otherwise we call $N$ is $x$-torsion free.
\par Now the elements $X$ and $\phi$
commute (see $(\ref{r4})$). Then there is a common eigenvector $v$ of the operators $X$ and $\phi$.
Put
\[vX=\lambda_1v,\ v\phi=\lambda_2v,\ \ \ \lambda_1,\lambda_2\in\mathbb{K}.\] Now depending on the scalars $\lambda_1$ and $\lambda_2$, we will consider the following cases.
\subsection{}\label{ss1} If $\lambda_1=0$ and $\lambda_2\neq 0$, then the action of $X$ on $N$ is trivial as $X$ is normal element. Hence $N$ becomes a simple module over the factor algebra $\mathcal{A}/\langle X\rangle$ which is isomorphic to a skew Laurent polynomial algebra $\mathbb{E}/\langle X \rangle[K^{\pm 1},\tau]$, where \[\mathbb{E}/\langle X \rangle\cong \mathbb{K}\langle E,Y~|~EY=q^{-1}YE\rangle\] and $\tau(E)=q^2E,~\tau(Y)=q^{-1}Y$ (cf. \cite[Eq $9,10$]{bavl}). Here $N$ becomes a simple module over the quantum affine space
\begin{equation}\label{qas}
  \mathbb{K}\langle E,Y,K~|~EY=q^{-1}YE,KE=q^2EK,KY=q^{-1}YK \rangle  
\end{equation} 
of rank $3$ with invertible action of the generator $K$.
\subsection{}\label{ss2} If $\lambda_1\neq 0$ and $\lambda_2=0$, then $N$ becomes a simple module over the factor algebra $\mathcal{A}/\langle \phi\rangle$ which is isomorphic to a skew Laurent polynomial algebra $\mathbb{E}/\langle \phi \rangle[K^{\pm 1},\tau]$, where \[\mathbb{E}/\langle \phi \rangle\cong \mathbb{K}\langle E,Y~|~EY=qYE\rangle\] and $\tau(E)=q^2E,~\tau(Y)=q^{-1}Y$ (cf. \cite[Eq $14,15$]{bavl}). In this case, $N$ becomes a simple module over the quantum affine space \[\mathbb{K}\langle E,Y,K~|~EY=qYE,KE=q^2EK,KY=q^{-1}YK \rangle\]
of rank $3$ with invertible action of the generator $K$. The simple modules over such quantum affine spaces have been classified already in \cite{smsb}.
\subsection{}\label{ss3} If both $\lambda_1$ and $\lambda_2$ are zero, then $N$ becomes a simple module over the the factor algebra $\mathcal{A}/\langle X, \phi\rangle$ which is isomorphic to a skew Laurent polynomial algebra $\mathbb{E}/\langle X, \phi\rangle[K^{\pm 1},\tau]$, where $\tau$ is same as  above and 
\[\mathbb{E}/\langle X, \phi\rangle \cong \mathbb{E}/\langle X\rangle/\langle X, \phi\rangle/\langle X \rangle \cong \mathbb{K}\langle E,Y~|~EY=q^{-1}YE\rangle/\langle EY \rangle\cong \mathbb{K}[E,Y]/\langle EY\rangle.\] Thus $N$ becomes a simple module over the commutative algebra $\mathbb{K}[E,Y]/\langle EY\rangle$ and hence the $\mathbb{K}$-dimension of $N$ is $1$.
\subsection{} In view of paragraphs $(\ref{ss1})-(\ref{ss3})$, henceforth we can assume that $\lambda_1\neq 0,\lambda_2\neq 0$, that is, $N$ is $X$ and $\phi$-torsion free simple $\mathcal{A}$-module. Now depending on $l$, we will consider two cases separately.\\
\subsubsection{\bf{Case A}} ($l$-odd)  Let $q$ be an odd primitive $l$-th root of unity. Since each of the monomials 
\begin{equation}\label{cev1}
   E^{l},\ Y^{l},\  K^{\pm l},\ X,\ \phi 
\end{equation}
commutes (see $(\ref{r4})$ and Lemma $\ref{c1}$), there is a common eigenvector $v$ of the monomials (\ref{cev1}).
Put
\[vE^l=\alpha v,\ vY^l=\beta v,\ vK^{\pm l}=\xi^{\pm 1}v, \ vX=\lambda_1v,\ v\phi=\lambda_2v,\]
for some $\alpha,\beta\in \mathbb{K}$ and $\xi,\lambda_1,\lambda_2\in \mathbb{K}^*$. Note that the central elements  $E^{l},\ Y^{l}$ and $K^{\pm l}$ act as scalars on $N$, by Schur's lemma. Then the following cases arise depending on scalars $\alpha$ and $\beta$:\\
\textbf{Case $1$:} ${\alpha\neq 0}$. Then the set $\{vE^{a_1}K^{a_2}~|~0\leq a_1,a_2\leq l-1\}$ consists non zero vectors of $N$. Let us choose
\[\mu_1:=\lambda_1,\  \mu_2:=\lambda_2,\ \mu_3:=\alpha^{\frac{1}{l}}\ \ \text{and}\ \ \mu_4:=\xi^{\frac{1}{l}}\]
so that $\underline{\mu}=(\mu_1,\mu_2,\mu_3,\mu_4)\in (\mathbb{K}^{*})^4$. Now we define an $\mathbb{K}$-linear map
\[\Phi:M_1(\underline{\mu}) \longrightarrow N\]
by specifying the image of basis vectors of $M_1(\underline{\mu})$ only
\begin{equation*}
\Phi\left(e(a_1,a_2)\right):=\mu_3^{-a_1}\mu_4^{-a_2} v E^{a_1}K^{a_2}.
\end{equation*} 
One can easily verify that $\Phi$ is a non zero ${\mathcal{A}}$-module homomorphism. In this verification the following computations will be very useful:
\[\begin{array}{l}
\left(vE^{a_1}K^{a_2}\right)E=\begin{cases}
 q^{2a_2}vE^{a_1+1}K^{a_2},&a_1\neq l-1\\
 \alpha q^{2a_2}vK^{a_2},&a_1= l-1
\end{cases}\\
\left(vE^{a_1}K^{a_2}\right)Y=\begin{cases}
 q^{-(a_1+a_2)}\left(\displaystyle\frac{q\lambda_2-q^{2a_1+1}\lambda_1}{1-q^2}\right)vE^{a_1-1}K^{a_2},& a_1\neq 0\\
  \alpha^{-1}q^{-a_2}\left(\displaystyle\frac{q(\lambda_2-\lambda_1)}{1-q^2}\right)vE^{l-1}K^{a_2},& a_1=0.
  \end{cases}
\end{array}\]
Thus by Schur's lemma, $\Phi$ is an isomorphism because $M_1(\underline{\mu})$ and $N$ are both simple ${\mathcal{A}}$-modules.\\
\textbf{Case $2$:} $\alpha=0$ and $\beta\neq 0$. Then there exists $0\leq r\leq l-1$ such that $v':=vE^{r}\neq 0$ and $v'E=0$. Therefore the set $\{v'Y^{a_1}K^{a_2}~|~0\leq a_1,a_2\leq l-1\}$ consists non zero vector of $N$. Now under the setting of scalar parameters 
\[\mu_1:=q^r\lambda_1,\  \mu_2:=\beta^{\frac{1}{l}}\ \ \text{and}\ \ \mu_3:=\xi^{\frac{1}{l}}\]
along with the $\mathbb{K}$-linear map $\Phi:M_2(\underline{\mu}) \longrightarrow N$  defined by
\begin{equation*}
\Phi\left(e(a_1,a_2)\right):=\mu_2^{-a_1}\mu_3^{-a_2} v'Y^{a_1}K^{a_2},
\end{equation*} one can easily verify that $\Phi$ is a non zero ${\mathcal{A}}$-module homomorphism between $M_2(\underline{\mu})$ and $N$. Infact the following computation will be useful for verification:
\[
\begin{array}{l}
\left(v'Y^{a_1}K^{a_2}\right)Y=\begin{cases}
 q^{-a_2}v'Y^{a_1+1}K^{a_2},&a_1\neq l-1\\
\beta q^{-a_2}v'K^{a_2},&a_1= l-1
\end{cases}\\
\left(v'Y^{a_1}K^{a_2}\right)E=\begin{cases}
 -\lambda_1q^{a_1+2a_2+r}\left(\displaystyle\frac{q^{-2a_1}-1}{q^{-2}-1}\right)v'Y^{a_1-1}K^{a_2},& a_1\neq 0\\
  0,& a_1=0.
\end{cases}
\end{array}\]
Thus by Schur's lemma $\Phi$ is an isomorphism.\\
\textbf{Case $3$:} ${\alpha=\beta=0}$. First consider the sequence of vectors \[v,vE,vE^2,\cdots,vE^{l-1},vE^l=0.\] Then there exists $0\leq r\leq l-1$ such that $v':=vE^{r}\neq 0$ and $vE^{r+1}=0$. Next consider another sequence of vectors \[v',v'Y,v'Y^2,\cdots,v'Y^{l-1},v'Y^l=0.\] Then there exists $1\leq s\leq l$ such that $v'Y^{s}=0$ and $v'Y^{s-1}\neq 0$. We now claim that $s=l$. Indeed,
$$\begin{array}{cl}
0=v'Y^sE&=v'q^s\left(EY^s-\displaystyle\frac{q^{-2s}-1}{q^{-2}-1}XY^{s-1}\right)\\
&=-q^s\displaystyle\frac{q^{-2s}-1}{q^{-2}-1}v'XY^{s-1}\\
&=-q^s\displaystyle\frac{q^{-2s}-1}{q^{-2}-1}\lambda_1q^{r-1}v'Y^{s-1}.
\end{array}$$
This implies $s$ is the smallest index such that $q^{-2s}-1=0$ and hence $s=l$. 
Thus $v'=vE^{r}(\neq 0)\in N$ is such that  
\[v'E=0,~v'Y^{l}=0,~v'Y^{l-1}\neq 0,~v'K^{\pm l}=\xi^{\pm 1} v',\]
\[~v'X=\lambda_1q^{r}v',~v'\phi=\lambda_1q^{r+2}v'.\]
Clearly the set $\{v'Y^{a_1}K^{a_2}~|~0\leq a_1,a_2\leq l-1\}$ consists non zero vectors of $N$. Then choose $\mu_1:=\lambda_1q^{r}$ and $\mu_2:=\xi^{\frac{1}{l}}$. Notice that $\underline{\mu}=(\mu_1,\mu_2)\in (\mathbb{K}^*)^2$. Now we define an $\mathbb{K}$-linear map
\[\Phi:M_3(\underline{\mu}) \longrightarrow N\]
by specifying the image of basis vectors of $M_3(\underline{\mu})$ only.
\begin{equation*}
\Phi\left(e(a_1,a_2)\right):=\mu_2^{-a_2} v' Y^{a_1}K^{a_2}.
\end{equation*} 
It is easy to verify that $\Phi$ is a non zero ${\mathcal{A}}$-module homomorphism. In this verification the following computations will be very useful:
\[\begin{array}{l}
\left(v'Y^{a_1}K^{a_2}\right)Y=\begin{cases}
 q^{-a_2}v'Y^{a_1+1}K^{a_2},& a_1\neq l-1\\
 0,& a_1=l-1
\end{cases}\\
\left(v'Y^{a_1}K^{a_2}\right)E=\begin{cases}
 -q^{a_1+2a_2+r}\lambda_1\displaystyle\frac{q^{-2a_1}-1}{q^{-2}-1}v'Y^{a_1-1}K^{a_2},& a_1\neq 0\\
 0,& a_1=0.
\end{cases}
\end{array}\]
Thus by Schur's lemma $\Phi$ is an isomorphism because $M_3(\underline{\mu})$ and $N$ are both simple ${\mathcal{A}}$-module.\\
\subsubsection{\bf{Case B}} ($l$-even) Let $q$ be an even primitive $l$-th root of unity. Here we can expect more commuting operators than odd case. By relation $(\ref{r4})$ and Lemma $\ref{c1}$, it follows that each of the monomials 
\begin{equation}\label{cev}
   E^{l},\ Y^{l},\  K^{\pm l},\ X,\ \phi,\ E^{\frac{l}{2}}K^{\frac{l}{2}},\ Y^{\frac{l}{2}}K^{\frac{l}{2}} 
\end{equation}
commutes. Therefore there is a common eigenvector $v$ of monomials (\ref{cev}).
Put
\begin{equation}
    vE^l=\alpha v,\ vY^l=\beta v,\ vK^{\pm l}=\xi^{\pm 1}v, \ vX=\lambda_1v,\ v\phi=\lambda_2v,
\end{equation}
\begin{equation}\label{ab}
    vE^{\frac{l}{2}}K^{\frac{l}{2}}=\alpha'v,~vY^{\frac{l}{2}}K^{\frac{l}{2}}=\beta'v,
\end{equation}
for some $\alpha,\alpha',\beta,\beta'\in \mathbb{K}$ and $\xi,\lambda_1,\lambda_2\in \mathbb{K}^*$. Now the relations in $(\ref{ab})$ along with its existing scalars can be expressed as
\begin{equation}\label{ev1}
vE^{\frac{l}{2}}=\xi^{-1}\alpha'vK^{\frac{l}{2}},~(\alpha')^{2}=\alpha \xi.
\end{equation}
\begin{equation}\label{ev2}
    vY^{\frac{l}{2}}=\xi^{-1}\beta'vK^{\frac{l}{2}},~(\beta')^{2}=q^{-\frac{l^2}{4}}\beta \xi.
\end{equation}
Note that the central elements  $E^{l},\ Y^{l}$ and $K^{\pm l}$ act as scalars on $N$, by Schur's lemma. Then the following cases arise:\\
\textbf{Case $1$:} ${\alpha\neq 0}$. Then by (\ref{ev1}), $\alpha'\neq 0$ and the set \[\{vE^{a_1}K^{a_2}~|~0\leq a_1\leq \frac{l}{2}-1,0\leq a_2\leq l-1\}\] consists non zero vectors of $N$. Let us choose
\[\mu_1:=\lambda_1,\  \mu_2:=\lambda_2,\ \mu_3:=\alpha^{\frac{1}{l}}\ \ \text{and}\ \ \mu_4:=\xi^{\frac{1}{l}}.\]
Notice that $\underline{\mu}=(\mu_1,\mu_2,\mu_3,\mu_4)\in (\mathbb{K}^{*})^4$. Now we define an $\mathbb{K}$-linear map
\[\Phi:M_1(\underline{\mu}) \longrightarrow N\]
by specifying the image of basis vectors of $M_1(\underline{\mu})$ only
\begin{equation*}
\Phi\left(e(a_1,a_2)\right):=\mu_3^{-a_1}\mu_4^{-a_2} v E^{a_1}K^{a_2}.
\end{equation*} 
It is easy to verify that $\Phi$ is a non zero ${\mathcal{A}}$-module homomorphism. In this verification the following computations will be very useful:
\[\begin{array}{l}
\left(vE^{a_1}K^{a_2}\right)E=\begin{cases}
q^{2a_2}vE^{a_1+1}K^{a_2},& a_1\neq \frac{l}{2}-1\\
 \xi^{-1}\alpha'q^{2a_2} vK^{\frac{l}{2}+a_2},& a_1=\frac{l}{2}-1.
\end{cases}\\
\left(vE^{a_1}K^{a_2}\right)Y=\begin{cases}
q^{-(a_1+a_2)}\left(\displaystyle\frac{q\mu_2-q^{2a_1+1}\mu_1}{1-q^2}\right)vE^{a_1-1}K^{a_2},& a_1\neq 0\\
 (\alpha')^{-1}q^{-a_2}\left(\displaystyle\frac{q(\mu_2-\mu_1)}{1-q^2}\right)vE^{\frac{l}{2}-1}K^{\frac{l}{2}+a_2},& a_1= 0.
\end{cases}
\end{array}\]
Thus by Schur's lemma $\Phi$ is an isomorphism because $M_1(\underline{\mu})$ and $N$ are both simple ${\mathcal{A}}$-module.\\
\textbf{Case $2$:} ${\alpha=0}$ and ${\beta\neq 0}$. Then by (\ref{ev1}) and (\ref{ev2}) we get $\alpha'=0$ and $\beta'\neq 0$. So there exists $0\leq r\leq \frac{l}{2}-1$ such that $v':=vE^{r}\neq 0$ and $v'E=0$. Therefore the set \[\{v'Y^{a_1}K^{a_2}~|~~0\leq a_1\leq \frac{l}{2}-1,0\leq a_2\leq l-1\}\] consists non zero vectors of $N$. Now under the setting of scalar parameters 
\[\mu_1:=q^r\lambda_1,\  \mu_2:=\beta^{\frac{1}{l}}\ \ \text{and}\ \ \mu_3:=\xi^{\frac{1}{l}}\]
along with the $\mathbb{K}$-linear map $\Phi:M_2(\underline{\mu}) \longrightarrow N$  define by
\begin{equation*}
\Phi\left(e(a_1,a_2)\right):=\begin{cases}
 \mu_2^{-a_1}\mu_3^{-a_2} v' Y^{a_1}K^{a_2},& a_1\neq 0\\
 q^{-\frac{l(r+1)}{2}}\mu_3^{-a_2}v'K^{a_2},&a_1=0
\end{cases}
\end{equation*} one can easily verify that $\Phi$ is a non zero ${\mathcal{A}}$-module homomorphism between $M_2(\underline{\mu})$ and $N$. To verify this one can use following simplifications:
\[\begin{array}{l}
\left(v'Y^{a_1}K^{a_2}\right)Y=\begin{cases}
 q^{-a_2}v'Y^{a_1+1}K^{a_2},&a_1\neq \frac{l}{2}-1\\
\xi^{-1}\beta' q^{-a_2-\frac{lr}{2}}v'K^{\frac{l}{2}+a_2}, &a_1=\frac{l}{2}-1
\end{cases}\\
\left(v'Y^{a_1}K^{a_2}\right)E=\begin{cases}
 -q^{a_1+2a_2}\left(\displaystyle\frac{q^{-2a_1}-1}{q^{-2}-1}\right)q^{r}\lambda_1v'Y^{a_1-1}K^{a_2},& a_1\neq 0\\
  0,& a_1=0.
\end{cases}
\end{array}
\]
Thus by Schur's lemma $\Phi$ is an isomorphism.\\
\textbf{Case $3$:} ${\alpha=\beta=0}$. Similar argument as in [$l$-odd-Case $3$], there exists $0\leq r\leq \frac{l}{2}-1$ such that  
\[v':=vE^{r}\neq 0,~v'E=0,~v'Y^{\frac{l}{2}}=0,~v'Y^{\frac{l}{2}-1}\neq 0,\] 
\[v'K^{\pm l}=\xi^{\pm 1}v',~v'X=\lambda_1q^{r}v',~v'\phi=\lambda_1q^{r+2} v'.\]
Clearly the set $\{v'Y^{a_1}K^{a_2}~|~0\leq a_1\leq \frac{l}{2}-1,0\leq a_2\leq l-1\}$ consists non zero vectors of $N$. Now under the setting of scalar parameters 
$\mu_1:=\lambda_1q^{r}$ and $\mu_2:=\xi^{\frac{1}{l}}$ along with the $\mathbb{K}$-linear map $\Phi:M_3(\underline{\mu}) \longrightarrow N$  define by
\begin{equation*}
\Phi\left(e(a_1,a_2)\right):=\mu_2^{-a_2} v' Y^{a_1}K^{a_2}.
\end{equation*} one can easily verify that $\Phi$ is a non zero ${\mathcal{A}}$-module homomorphism between $M_3(\underline{\mu})$ and $N$. Thus by Schur's lemma $\Phi$ is an isomorphism.
\par Finally the above discussions lead us to the main result of this section which provides an opportunity for classification of simple $\mathcal{A}$-modules in terms of scalar parameters.
\begin{theom}
Let $q$ be a primitive $l$-th root of unity with $l\geq 3$. Then each $X,\phi$-torsion free simple $\mathcal{A}$-module is isomorphic to one of the following simple $\mathcal{A}$-modules: 
\begin{enumerate}
    \item $M_1(\underline{\mu})$ for some $\underline{\mu}\in (\mathbb{K}^*)^4$;
    \item $M_2(\underline{\mu})$ for some $\underline{\mu}\in (\mathbb{K}^*)^4$;
    \item $M_3(\underline{\mu})$ for some $\underline{\mu}\in (\mathbb{K}^*)^2$.
\end{enumerate}
Moreover the $\mathbb{K}$-dimension of each $X,\phi$-torsion free simple $\mathcal{A}$-module is maximal which is equal to $\pideg (\mathcal{A})$.
\end{theom}
\section{\bf{Finite Dimensional Indecomposable $\mathcal{A}$-Modules}}
In this section we aim to construct some finite dimensional indecomposable modules over $\mathcal{A}$. Let $q$ be a primitive $l$-th root of unity. Let $\mathcal{B}$ be the subalgebra of $\mathcal{A}$ generated by the $E,X$ and $K^{\pm l}$. Observe that $\mathcal{B}E=E\mathcal{B}$ is an ideal of $\mathcal{B}$ and $\mathcal{B}/\mathcal{B}E\cong\mathbb{K}[X,K^{\pm l}]$ is a commutative algebra. Then for $(\lambda_1,\lambda_2)\in(\mathbb{K}^*)^2$, there is a one dimensional $\mathcal{B}$-module $\mathbb{K}(\lambda_1,\lambda_2)$ given by
\[vE=0,\ vX=\lambda_1v,\ vK^{\pm l}=\lambda_2^{\pm 1}v\]
for $v\in \mathbb{K}=\mathbb{K}(\lambda_1,\lambda_2)$. Define the Verma module $M(\lambda_1,\lambda_2)$ by
\[M(\lambda_1,\lambda_2):=\mathbb{K}(\lambda_1,\lambda_2)\otimes_{\mathcal{B}}\mathcal{A}.\]
Since $\mathcal{A}$ is a free left $\mathcal{B}$-module with basis \[\{Y^mK^n~|~0\leq n\leq l-1, m\geq 0\},\]
by Proposition \ref{imp}, $M(\lambda_1,\lambda_2)$ has a vector space basis \[f(m,n):=v \otimes Y^mK^n,\ \ 0\leq n\leq l-1, m\geq 0.\]
Thus $M(\lambda_1,\lambda_2)$ carries an infinite dimensional right $\mathcal{A}$-module. The explicit action of the generators of $\mathcal{A}$ on $M(\lambda_1,\lambda_2)$ is given by
\[
\begin{array}{l}
f(m,n)K^{\pm 1}=f(m,n\pm 1)\\
f(m,n)X=\lambda_1q^{n-m}f(m,n)\\
f(m,n)Y=q^{-n}f(m+1,n)\\
f(m,n)E=\begin{cases}
 -q^{m+2n}\displaystyle\frac{q^{-2m}-1}{q^{-2}-1}f(m-1,n),& m>0\\
 0,& m=0
\end{cases}
\end{array}\]
If $l$ is odd, then one gets the invariant subspaces
\[M_{p,l}:=\spa\{f(pl+i,n)~|~i\geq 0,0\leq n\leq l-1\},\ \ \ \ p=1,2,\cdots\]
owing to the equation $f(pl,n)E=0$. Evidently one has
\[M(\lambda_1,\lambda_2)\supset M_{1,l}\supset M_{2,l} \supset M_{3,l}\supset\cdots\]
and \[Q_{1,l}\subset Q_{2,l}\subset Q_{3,l}\subset \cdots,\]
where $Q_{p,l}=M(\lambda_1,\lambda_2)/M_{p,l}$ is the quotient space corresponding to $M_{p,l}$. Also observe that
\[Q_{p,l}=\spa\{\overline{f}(m,n)=f(m,n) \mo M_{p,l}~|~0\leq m\leq pl-1,0\leq n\leq l-1\}.\]
On each $Q_{p,l}$, the above action induces $pl^2$-dimensional $\mathcal{A}$-module structure:
\[
\begin{array}{l}
\overline{f}(m,n)K^{\pm 1}=\overline{f}(m,n\pm 1)\\
\overline{f}(m,n)X=\lambda_1q^{n-m}\overline{f}(m,n)\\
\overline{f}(m,n)Y=\begin{cases}
 q^{-n}\overline{f}(m+1,n),& 0\leq m< pl-1\\
 0,&m=pl-1
\end{cases}\\
\overline{f}(m,n)E=\begin{cases}
 -q^{m+2n}\displaystyle\frac{q^{-2m}-1}{q^{-2}-1}\overline{f}(m-1,n),& 0<m\leq pl-1\\
 0,& m=0
\end{cases}
\end{array}\]
Next we will focus on the $\mathcal{A}$-submodules of $Q_{p,l}$. 
\begin{lemma}\label{submod}
The $\mathcal{A}$-submodules of $Q_{p,l}$ are of the form $M_{r,l}/M_{p,l}, 1\leq r\leq p$ or $M(\lambda_1,\lambda_2)/M_{p,l}$.
\end{lemma}
\begin{proof}
It is enough to show that the $\mathcal{A}$-submodules of $M(\lambda_1,\lambda_2)$ containing $M_{p,l}$ are of the form $M_{r,l}, 1\leq r\leq p$ or $M(\lambda_1,\lambda_2)$. Suppose $W$ be a $\mathcal{A}$-submodule containing $M_{p,l}$. If $W=M_{p,l}$, then we are done. Otherwise one can choose $0\neq w\in W$ such that
\[w=\sum_{n=0}^{l-1}\sum_{m=0}^{pl-1}C_{mn}f(m,n),\ \ C_{mn}\in \mathbb{K},\] with at least one non zero $C_{mn}$. Let $m'=\mi\{m~|~C_{mn}\neq 0\}$. Then $wY^{pl-1-m'}\in W$. Now
\[
\begin{array}{cl}
wY^{pl-1-m'}&=\sum\limits_{n=0}^{l-1}\sum\limits_{m=m'}^{pl-1}C_{mn}q^{-n(pl-1-m')}f(m+pl-1-m',n)\\
&=w'+\sum\limits_{n=0}^{l-1}C_{m'n}q^{-n(pl-1-m')}f(pl-1,n),
\end{array}\]
where $w'=\sum\limits_{n=0}^{l-1}\sum\limits_{m=m'+1}^{pl-1}C_{mn}q^{-n(pl-1-m')}f(m+pl-1-m',n) \in M_{p,l}$. Thus \[w''=wY^{pl-1-m'}-w'=\sum\limits_{n=0}^{l-1}C_{m'n}q^{-n(pl-1-m')}f(pl-1,n)\in W.\] Now if $C_{m'n_1}$ and $C_{m'n_2}$ are two non zero scalars in $w''$, then $w''X-\lambda_1q^{n_1-pl+1}w''$ is a non zero element in $W$ smaller length than $w''$. Hence by
induction it follows that $f(pl-1,n')\in W$ for some $0\leq n'\leq l-1$. Finally with the action of $E$ and $K^{\pm 1}$ on $f(pl-1,n')$, we have $M_{p-1,l}\subseteq W$. Now if $W=M_{p-1,l}$, then we are done. Otherwise continuing with the above argument one can obtain the desired result.
\end{proof}
With this we have the following:
\begin{theom}
Suppose $l$ is odd. Then
\begin{enumerate}
    \item [(1)] $Q_{1,l}$ is a simple $\mathcal{A}$-module.
    \item [(2)] When $p>1$, then $Q_{p,l}$ is neither simple nor semisimple $\mathcal{A}$-module.
    \item [(3)] When $p>1$, then $Q_{p,l}$ is an indecomposable $\mathcal{A}$-module.
\end{enumerate}
\end{theom}
\begin{proof}
(1) It follows from Lemma \ref{submod}.\\
(2) First of all $Q_{p,l}$ is semisimple if and only if for each invariant subspace $W\subset Q_{p,l}$, there is an invariant subspace $\overline{W}$ complementary to it. Now let us consider the subspace
\[W=\spa\{\overline{f}(m,n)~|~(p-1)l\leq m\leq pl-1,0\leq n\leq l-1\}.\]
It is easily seen that $W$ is $\mathcal{A}$-invariant. One can prove that it does not have an invariant complementary subspace $\overline{W}$ in $Q_{p,l}$. Otherwise a non zero element $x\in \overline{W}$ can be written as
\[x=\sum\limits_{n=0}^{l-1}\sum\limits_{m=0}^{(p-1)l-1}C_{mn}\overline{f}(m,n)+\sum\limits_{n=0}^{l-1}\sum\limits_{m=(p-1)l}^{pl-1}D_{mn}\overline{f}(m,n),\] where there is at least one non zero $C_{mn}$. Let $m'=\mi\{m~|~C_{mn}\neq 0\}$. Then we have $xY^{(p-1)l-m'}\in W$. This contradicts the assumption that $\overline{W}$ is $\mathcal{A}$-invariant.\\
(3) Recall that a module is indecomposable if it is non zero and cannot be written as a direct sum of two non zero submodules. With the argument in (2) one can easily prove that any of the non zero proper $\mathcal{A}$-invariant subspaces $M_{r,l}/M_{p,l},1\leq r\leq p-1$ in $Q_{p,l}$, there is no invariant subspace complementary to it. Thus when $p>1$, the module $Q_{p,l}$ is a finite dimensional indecomposable $\mathcal{A}$-module. 
\end{proof}
\begin{remak}
If $q$ is an even primitive $l$-th root of unity, then one can make similar discussion. Here one can get the $\mathcal{A}$-invariant subspaces 
\[M_{p,\frac{l}{2}}:=\spa\{f\left(\frac{pl}{2}+i,n\right)~|~i\geq 0,0\leq n\leq l-1\},\ \ \ \ p=1,2,\cdots\]
of $M(\lambda_1,\lambda_2)$, owing to the equation $f\left(\frac{pl}{2},n\right)E=0$.
\end{remak}

\section{\bf{Simple modules over $\mathbb{K}_q[X,Y]\rtimes U_q(\mathfrak{sl}_2)$}}
In this section, we will consider the smash product algebra $A:=\mathbb{K}_q[X,Y]\rtimes U_q(\mathfrak{sl}_2)$ (cf. \cite{bavlu}). As an abstract algebra, the algebra $A$ is generated over $\mathbb{K}$ by elements $E,F,K,K^{-1},X$ and $Y$ subject to the defining relations (where $K^{-1}$ is the inverse of $K$):
\[KEK^{-1}=q^2E,~ KFK^{-1}=q^{-2}F,~EF-FE=\frac{K-K^{-1}}{q-q^{-1}},\]
\[EX=qXE,~EY=X+q^{-1}YE,~FX=YK^{-1}+XF,~FY=YF,\]
\[KXK^{-1}=qX,~KYK^{-1}=q^{-1}Y,~XY=qYX.\]
Using the defining relations of the algebra A, the quantum spatial ageing algebra $\mathcal{A}$ is a subalgebra of $A$ and the algebra $A$ is a skew polynomial algebra $A=\mathcal{A}[F,\sigma,\delta]$, where $\sigma$ is an automorphism of $\mathcal{A}$ such that $\sigma(K)=q^2K,\sigma(E)=E,\sigma(X)=X,\sigma(Y)=Y$ and $\delta$ is a $\sigma$-derivation of the algebra $\mathcal{A}$ such that $\delta(K)=0,\delta(E)=\frac{K-K^{-1}}{q-q^{-1}},\delta(X)=YK^{-1},\delta(Y)=0$ (cf. \cite{bavlu}, eq(2.18), eq(2.19)).
\par In this section we aim to classify simple $A$-modules on which $X$ and $Y$ acting as invertible operator assuming that $\mathbb{K}$ is an algebraically closed field and $q$ is a primitive $l$-th root of unity with $l\geq 3$.
\par Let $T$ be the subalgebra of $A$ generated by the elements $K^{\pm 1},X$ and $Y$. Clearly, $T:=\mathbb{K}_{q}[X,Y][K^{\pm 1},{\tau}]$ and $\tau(X)=qX,\tau(Y)=q^{-1}Y$. Let us consider two elements \[\phi:=EY-qYE\ \ \text{and}\ \ \psi:=XF-q^2FX\] of the algebra $A$ (cf. \cite{bavlu}). Observe that $\phi$ is a normal element of $\mathcal{A}$ (see (\ref{r4})) but not a normal of $A$ (cf. \cite{bavlu}, Lemma 2.6).
Let $B$ be the subalgebra of $A$ generated by the algebra $T$ and the elements $\phi$ and $\psi$. The generators $K^{\pm 1},X,Y,\phi$ and $\psi$ of $B$ satisfy the following relations:
\[\phi X=X\phi,~\phi Y=q^{-1}Y\phi,~\phi K=q^{-1}K\phi,\]
\[\psi X=X\psi,~\psi Y=qY\psi,~\psi K=qK\psi,\]
\[\psi\phi-\phi\psi=q(1-q^2)KYX.\]
In particular, the algebra $B=T[\phi][\psi,\alpha,\delta]$ is skew polynomial algebra (cf. \cite{bavlu}), where $\alpha$ is an automorphism of $T[\phi]$ such that $\alpha(X)=X,~\alpha(Y)=qY,~\alpha(K)=qK,~\alpha(\phi)=\phi$ and $\delta$ is a $\alpha$-derivation of $T[\phi]$ such that $\delta(X)=\delta(Y)=\delta(K)=0$ and $\delta(\phi)=q(1-q^2)KYX$. It is easy to observe that the elements $K^{\pm 1}, X$ and $Y$ are normal in $B$ but not in $A$. 
\subsection{PI Algebras $A$ and $B$}
Here we will use Proposition \ref{f}  for proving $A$ and $B$ are PI-algebra. The following lemmas will be very useful.
\begin{lemm}\label{acop}
\begin{enumerate}
    \item The following identities hold in the algebra $A$:
    \begin{enumerate}
        \item [(i)] $FX^s=X^sF+\displaystyle\frac{1-q^{2s}}{1-q^2}YK^{-1}X^{s-1}$
        \item [(ii)] $XF^r=F^rX-\displaystyle\frac{1-q^{2r}}{1-q^2}YF^{r-1}K^{-1}$
        \item [(iii)] $EF^s=F^sE+\displaystyle\frac{q^s-q^{-s}}{q-q^{-1}}F^{s-1}\displaystyle\frac{Kq^{1-s}-K^{-1}q^{s-1}}{q-q^{-1}}$
        \item [(iv)] $FE^r=EF^r-\displaystyle\frac{q^r-q^{-r}}{q-q^{-1}}E^{r-1}\displaystyle\frac{Kq^{r-1}-K^{-1}q^{1-r}}{q-q^{-1}}$
        \item [(v)] $\psi^s \phi =\phi \psi^s+q(1-q^{2s})KYX\psi^{s-1}$
        \item [(vi)] $\psi \phi^r =\phi^r \psi+q^3(q^{-2r}-1)KYX\phi^{r-1}$.
    \end{enumerate}
\end{enumerate}
\end{lemm}
\begin{lemm}\label{it}
Let $q$ be a primitive $l$-th root of unity. Then
\begin{enumerate}
    \item $K^{\pm l},E^{l},F^{l},X^{l}$ and $Y^{l}$ are contained in the center of ${A}$.
    \item $K^{\pm l},X^{l},Y^{l},\phi^{l}$ and $\psi^{l}$ are contained in the center of ${B}$.
\end{enumerate}
These two results will be used throughout this section.
\end{lemm}
\begin{prop} \label{pia}
The algebra ${A}$ (respectively, $B$) is a PI algebra if and only if $q$ is a root of unity.
\end{prop}
\begin{proof}
The argument for this is same as Proposition $\ref{finite}$ using Lemma $\ref{it}$.
\end{proof}
Now we aim to compute an explicit expression of PI-deg for $\mathcal{A}$ in the following. First using the equivalent expression of the elements \begin{equation}\label{eqex}
    \phi=X-(q-q^{-1})YE\ \ \ \text{and}\ \ \  \psi=(1-q^2)XF-q^2YK^{-1}
\end{equation} we see that 
\begin{equation}\label{spi}
    A_{X,Y}=B_{X,Y}
\end{equation} where $A_{X,Y}$ and $B_{X,Y}$ are the localization
of the algebras $A$ and $B$ respectively at the powers of the elements $X$ and $Y$ (cf. \cite{bavlu}, eq(2.12)). Next observe that \[B\subset {B}_{X,Y}\subset \cf{B}\ \ \text{and}\ \ A\subset {A}_{X,Y}\subset \cf{A}\] and hence it follows from \cite[Corollary I.13.3]{brg} and relation (\ref{spi}) that \[\pideg(B)=\pideg({B}_{X,Y})=\pideg({A}_{X,Y})=\pideg({A}).\] Note that the skew relation $\alpha \delta=q^2\delta \alpha,~(q^2\neq 1)$ is satisfied on $T[\phi]$. Then the derivation erasing process (independent of characteristic) in \cite[Theorem 7]{lm2} provides $\cf {{B}}\cong \cf \mathcal{O}_{\mathbf{q}}(\mathbb{K}^{5})$ and hence $\pideg(B)=\pideg \mathcal{O}_{\mathbf{q}}(\mathbb{K}^5)$, where the $(5\times 5)$ matrix of relations $\mathbf{q}$ is \[\mathbf{q}=\begin{pmatrix}
1&q&q^{-1}&1&1\\
q^{-1}&1&q&q&q^{-1}\\
q&q^{-1}&1&q&q^{-1}\\
1&q^{-1}&q^{-1}&1&1\\
1&q&q&1&1
\end{pmatrix}.\]
The integral matrix associated to $\mathbf{q}$ has rank $4$ and invariant factors $1,1,2,2$. Therefore from above discussions along with the help of Lemma \ref{mainpi} we have $$\text{PI deg}~{(A)}=\text{PI deg}~{(B)}=\pideg \mathcal{O}_{\mathbf{q}}(\mathbb{K}^5)=
\begin{cases}
 l^2,&  l~ \text{odd}\\
 \frac{l^2}{2}, &  l~ \text{even}
\end{cases}$$
\subsection{Relationship between simple $A$-module and simple $B$-module}
Here we wish to establish a connection between simple $A$-modules and simple $B$-modules. From Proposition $\ref{sim}$ along with Proposition $\ref{pia}$, it is quite clear that each simple $A$-module is finite dimensional and can have dimension at most $\pideg(A)$. Same is true for the algebra $B$ also.
\par Let $N$ be a $X,Y$-torsion free simple $B$-module. Then the operators $X$ and $Y$ on $N$ are invertible. With this fact and the equivalent expression of $\phi$ and $\psi$ in (\ref{eqex}), one can define the action of $E$ and $F$ on $N$ explicitly. Thus $N$ becomes a $A$-module which is simple as well. Hence any $X,Y$-torsion free simple $B$-module is a simple $A$-module with invertible action of $X$ and $Y$.
\par On the other hand suppose $M$ be a simple $A$-module with invertible action of $X$ and $Y$ on $M$. Now we claim that $M$ is a $X,Y$-torsion free simple $B$-module. Clearly $M$ is a finite dimensional $B$-module. Suppose $M'$ be a non zero simple $B$-submodule of $M$. As $X$ and $Y$ are invertible operator on $M'$, therefore one can define the action of $E$ and $F$ on $M'$. Thus $M'$ becomes a $A$-module with $M'\subseteq M$. Therefore $M=M'$ and $M$ is a $X,Y$-torsion free simple $B$-module. Thus we have proved that
\begin{theo}\label{itd}
There is an one to one correspondence between $X,Y$-torsion free simple $B$-modules and simple $A$-modules with invertible action of $X$ and $Y$.
\end{theo}
In the next subsection we aim to classify $X,Y$-torsion free simple $B$-modules, which are quite similar to the subsection $5.4$ for the algebra $\mathcal{A}$.
\subsection{Classification of $X,Y$-torsion free simple $B$-modules} Let $N$ be a $X,Y$-torsion free simple $B$-module. Now depending on $l$, we will consider two cases separately.\\
\subsubsection{\bf{Case A}} ($l$-odd) Let $q$ be an odd primitive $l$-th root of unity. Since each of the elements
\begin{equation}\label{opc1}
    \phi^l,\psi^l,X^l,K,YX,\psi\phi
\end{equation} commutes and $N$ is a finite dimensional vector space over $\mathbb{K}$, there is a common eigen vector $v$ of the operators in (\ref{opc1}). Put
\[
\begin{array}{l}
   v\phi^l=\alpha v,~v\psi^l=\beta v,~vX^l=\xi v,\ \ \alpha,\beta,\xi\in\mathbb{K} \\
   vK=\lambda_1 v,~vYX=\lambda_2 v,~v\psi\phi=\lambda_3 v,\ \ \lambda_i\in \mathbb{K}.
\end{array}
\]
Clearly $\xi,\lambda_1,\lambda_2\in \mathbb{K}^{*}$. In view of Lemma \ref{it} the central elements act as multiplication by scalar on $N$ by Schur’s lemma.
Then the following cases arise depending on scalars $\alpha$ and $\beta$:\\
\textbf{Case 1:} $\alpha\neq 0$. Consider the vector subspace $N_1$ of $N$ generated by the non zero vectors $e(a,b):=v\phi^aX^b$, where $0\leq a,b\leq l-1$. We claim that $N_1$ is a $B$-submodule. In fact after some direct calculation we get
\[\begin{array}{l}
    e(a,b)K^{\pm 1}=q^{-b-a}\lambda_1e(a,b)\\
    e(a,b)\phi=e(a+1,b)\\
       e(a,b)\psi=\begin{cases}
        (\lambda_3-q^3(q^{-2a}-1)\lambda_1\lambda_2)e(a-1,b),& a\neq 0\\
        \alpha^{-1}\lambda_3e(l-1,b),& a=0
       \end{cases}\\
       e(a,b)X=e(a,b+1)\\
       e(a,b)Y=\begin{cases}
        q^{b-a}\lambda_2e(a,b-1),& b\neq 0\\
        \xi^{-1}q^{-a}\lambda_2e(a,l-1),& b=0
       \end{cases}
       \end{array}\]
Therefore owing to simpleness of $N$, $N=N_1$. Now the following result deciphers the dimension of $N_1$.
\begin{theo}\label{dimag}
The $B$-module $N_1$ has dimension $l^2$.
\end{theo}
\begin{proof}
Note that the vector $e(a,b)\in N_1$ is an eigenvector of $K$ and $YX$ associated with the eigenvalue $\lambda_1q^{-b-a}$ and $\lambda_2q^{b-a}$ respectively. Using the same argument as in Theorem $\ref{f1}$, we can obtain the desired result.
\end{proof}
~\\
\textbf{Case 2:} $\alpha= 0$ and $\beta\neq 0$. Consider the vector subspace $N_2$ of $N$ generated by the non zero vectors $e(a,b):=v\psi^aX^b$, where $0\leq a,b\leq l-1$. Using the similar argument as above, we conclude that $N_2$ is $B$-invariant submodule of $N$. Hence $N_2$ is a simple $B$-module of dimension $l^2$ (similar argument as in Theorem \ref{dimag}).\\
\textbf{Case 3:} $\alpha= 0$ and $\beta=0$.
Let $r$ be the smallest integer with $1 \leq r \leq m$ such that  $v\phi^r=0$ and $v\phi^{r-1} \neq 0$. Define $u:=v\phi^{r-1}$. Let $s$ be the smallest integer with $1 \leq s \leq m$ such that $u\psi^s=0$ and $u\psi^{s-1} \neq 0$. We now claim that $s=l$. Indeed,
$$\begin{array}{cl}
0=u\psi^s\phi&=u\left(\phi\psi^s+q(1-q^{2s})KYX\psi^{s-1}\right)\\
&=q(1-q^{2s})uKYX\psi^{s-1}\\
&=q(1-q^{2s})q^{-2r+2}\lambda_1\lambda_2u\psi^{s-1} 
\end{array}$$
This implies $s$ is the smallest index such that $q^{2s}-1=0$ and hence $s=l$. 
Thus $u=v\phi^{r-1}(\neq 0)\in N$ is such that  
\[u\phi=0,~u\psi^{l}=0,~u\psi^{l-1}\neq 0,~uX^l=\xi u,\]
\[~uK=q^{-r+1}\lambda_1u,~uYX=q^{-r+1}\lambda_2u, u\psi \phi=q(1-q^{2})q^{-2r+2}\lambda_1\lambda_2u.\]
Now the vector subspace $N_3$ of $N$ generated by the non zero vectors $e(a,b):=u\psi^aX^b$, where $0\leq a,b\leq l-1$ is a $B$-invariant submodule of $N$. Therefore $N_3$ is a simple $B$-module of dimension $l^2$ (similar argument as in Theorem \ref{dimag}).\\
\subsubsection{\bf{Case B}} ($l$-even) Let $q$ be an even primitive $l$-th root of unity. Here one can expect more commuting operators than odd case. By the defining relation of $B$ along with Lemma \ref{acop}, it follows that each of the elements 
\begin{equation}\label{opc2}
   \phi^l,\psi^l,X^l,K,YX,\psi\phi, \phi^{\frac{l}{2}}X^{\frac{l}{2}},\ \psi^{\frac{l}{2}}X^{\frac{l}{2}} 
\end{equation}
commutes. Therefore there is a common eigenvector $v$ of the operators (\ref{opc2}) on $N$.
Put
\begin{equation}
    v\phi^l=\alpha v,~v\psi^l=\beta v,~vX^l=\xi v,
\end{equation}
\begin{equation}
   vK=\lambda_1 v,~vYX=\lambda_2 v,~v\psi\phi=\lambda_3 v,\ \ \end{equation}
\begin{equation}\label{ab1}
    v\phi^{\frac{l}{2}}X^{\frac{l}{2}}=\alpha'v,~v\psi^{\frac{l}{2}}X^{\frac{l}{2}}=\beta'v,
\end{equation}
for some $\alpha,\alpha',\beta,\beta'\in \mathbb{K}$ and $\xi,\lambda_1,\lambda_2\in \mathbb{K}^*$. Now the relations in $(\ref{ab1})$ along with its existing scalars can be expressed as
\begin{equation}\label{evv1}
v\phi^{\frac{l}{2}}=\xi^{-1}\alpha'vX^{\frac{l}{2}},~(\alpha')^{2}=\alpha \xi.
\end{equation}
\begin{equation}\label{evv2}
    v\psi^{\frac{l}{2}}=\xi^{-1}\beta'vX^{\frac{l}{2}},~(\beta')^{2}=\beta \xi.
\end{equation}
Then the following cases arise:\\
\textbf{Case 1:} $\alpha\neq 0$. Then $\alpha'\neq 0$ and the span $N_4$ of the set of non zero vectors \[\{v\phi^{a}X^{b}~|~0\leq a\leq \frac{l}{2}-1,0\leq b\leq l-1\}\] is $B$-invariant submodule of $N$. Thus $N_4$ is a simple $B$-module of dimension $\frac{l^2}{2}$ (similar argument as in Theorem \ref{f1}).\\
\textbf{Case 2:} $\alpha= 0$ and $\beta\neq 0$. Then $\alpha'=0$ and $\beta'\neq 0$. So there exists $0\leq r\leq \frac{l}{2}-1$ such that $u:=v\phi^{r}\neq 0$ and $u\phi=0$. Now after some direct calculation the span $N_5$ of the set \[\{u\psi^{a}X^{b}~|~~0\leq a\leq \frac{l}{2}-1,0\leq b\leq l-1\}\] is $B$-invariant submodule of $N$. Therefore $N_5$ is a simple $B$-module of dimension $\frac{l^2}{2}$ (similar argument as in Theorem \ref{f1}). \\
\textbf{Case 3:} $\alpha= 0$ and $\beta=0$. Similar argument as in [$l$-odd-Case $3$], there exists $0\leq r\leq \frac{l}{2}-1$ such that  
\[u:=v\phi^{r}\neq 0,~u\phi=0,~u\psi^{\frac{l}{2}}=0,~u\psi^{\frac{l}{2}-1}\neq 0,~uX^l=\xi u,\] 
\[~uK=q^{-r}\lambda_1u,~uYX=q^{-r}\lambda_2u, u\psi \phi=q(1-q^{2})q^{-2r}\lambda_1\lambda_2u.\]
Now the vector subspace $N_6$ of $N$ generated by the vectors $e(a,b):=v\psi^aX^b$, where $0\leq a\leq \frac{l}{2}$, $b\leq l-1$ is $B$-invariant. Hence $N_6$ is a simple $B$-module of dimension $\frac{l^2}{2}$ (similar argument as in Theorem \ref{f1}).
\par Finally the above discussions along with Theorem \ref{itd} lead us to the main result of this section.
\begin{theom}
Let $q$ be a primitive $l$-th root of unity with $l\geq 3$. Then each simple ${A}$-module with invertible action of $X$ and $Y$ is isomorphic to one of the simple $B$-module $N_i$, $1\leq i\leq 6$ as mentioned above. Moreover the $\mathbb{K}$-dimension each such simple $A$-module is maximal which is equal to $\pideg {A}$.
\end{theom}

\end{document}